\documentclass[a4paper,fleqn]{article}
\usepackage{t1enc}
\usepackage[english]{babel}
\usepackage{amsmath}
\usepackage{amssymb}
\usepackage{color}
\usepackage{amsfonts}
\usepackage{enumerate}
\begin{document}

%%%%%%%%%%%%%%%%%%%%%%%%% enunciats %%%%%%%%%%%%%%%%%%%%%%%

\renewcommand{\theequation}{\thesection.\arabic{equation}}

 \newtheorem{theorem}{Theorem}[section]
 \newtheorem{prop}[theorem]{Proposition}
 \newtheorem{lemma}[theorem]{Lemma}
 \newtheorem{defin}[theorem]{Definition}
 \newtheorem{corollary}[theorem]{Corollary}
 \newtheorem{property}[theorem]{}
 \newtheorem{remark}[theorem]{Remark}
\newtheorem{example}[theorem]{Example}
\newtheorem{examples}[theorem]{Examples}
\newtheorem{claim}{Claim}

%%%%%% Enumeracions %%%%%%
%\newcommand{\seccion}[1]{\section{#1} \setcounter{equation}{0}}
%\renewcommand{\theequation}{\thesection.\arabic{equation}}

%%%%%%%%%% definicions %%%%%%%

\newcommand{\fun}[3]{\mbox{$#1\colon #2 \rightarrow #3$}}
\newcommand{\group}{$\ell$-group}
\newcommand{\ideal}{$\ell$-ideal}
\newcommand{\smsm}{\smallsetminus}
\newcommand{\bg}{\begin}
\newcommand{\e}{\end}
\newcommand{\ee}{\e{enumerate}}
\newcommand{\bal}{\bg{align*}}
\newcommand{\eal}{\end{align*}}
\newcommand{\bd}{\bg{description}}

\newcommand{\ed}{\e{description}}
\newcommand{\be}{\bg{enumerate}}
\newcommand{\ei}{\e{itemize}}
\newcommand{\beq}{\bg{eqnarray*}}
\newcommand{\eeqn}{\e{eqnarray}}
\newcommand{\beqn}{\bg{eqnarray}}
\newcommand{\eeq}{\e{eqnarray*}}
\newcommand{\bequ}{\bg{equation}}
\newcommand{\eequ}{\e{equation}}
\newcommand{\bi}{\bg{itemize}}
\newcommand{\mbf}[1]{\boldsymbol{#1}}
\newcommand{\Ra}{\Rightarrow}
\newcommand{\ra}{\rightarrow}
\newcommand{\La}{\Leftarrow}
\newcommand{\Lra}{\Leftrightarrow}
\newcommand{\rdot}{\overline{\odot}}
\newcommand{\igual}[1]{[\![#1]\!]}
\newcommand{\proof}{{\em Proof:\/}}
\newcommand{\alg}{BL-algebra}
\newcommand{\no}{\noindent}
\newcommand{\cqd}{\hfill$\Box$\medskip}
\newcommand{\mr}[1]{\buildrel {#1} \over \longrightarrow}
\newcommand{\ml}[1]{\buildrel {#1} \over \longleftarrow}
\newcommand{\cc}{\mathcal}
\newcommand{\A}{\mbf{A}}
\newcommand{\B}{\mbf{B}}
\newcommand{\sh}{\mathfrak}
\newcommand{\bb}{\mathbb}
\newcommand{\US}{\langle U_S \rangle}
\newcommand{\F}{\mathfrak{F}}
\newcommand{\Vt}[1]{\bb{V}_{\mbf{#1}}}
\newcommand{\Vna}[1]{\bb{V}_{\mathbf{\nabla}_{#1}}}
\newcommand{\free}[2]{\F_{\bb{#1}}(#2)}
\newcommand{\lukn}{$\widehat{\mbf{L_{n+1}}}$}
\newcommand{\dn}{\neg\neg}
\newcommand{\D}{\mbf{D}}
\newcommand{\T}{\mathbf{T}}
\newcommand{\C}{\mathbf{C}}
\newcommand{\Pa}{\mathbf{P}}
\newcommand{\Sa}{\mathbf{S}}

\newcommand{\sideremark}[1]{\marginpar{\scriptsize
\flushleft{\hspace{0pt}\fbox{\fbox{\parbox{2.5cm}{#1}}}}}}
\newcommand{\leftsideremark}[1]{\scriptsize
\flushleft{\hspace{0pt}\fbox{\fbox{\parbox{2.5cm}{#1}}}} }

%%%%%%%%%%%%%%%%%%%%%%%%%%%%

 \title{A categorical equivalence 
for Stonean residuated lattices}
 \author{\textsc{M. Busaniche,  R. Cignoli
  and M. Marcos}}

 \date{}

 \maketitle

\abstract{Distributive Stonean residuated lattices are closely related to Stone algebras since their bounded lattice reduct is a Stone algebra. In the present work we follow the ideas presented by Chen and Gr\"{a}tzer and try to apply them for the case of Stonean residuated lattices. Given a Stonean residuated lattice, we consider the triple formed by its Boolean skeleton, its algebra of dense elements and a connecting map.  We define a category whose objects are these triples and suitably defined morphisms, and prove that we have a categorical equivalence between this category and that of Stonean residuated lattices. We compare our results with other works and show some applications of the equivalence.}

 \section*{Introduction and overview}
%  usual sectioning commands i.e.,
%  \section, \subsection, \section*, \subsection*, etc are
%  allowed.

%This paper explores relationships between certain \textit{something} and
%\textit{something else}.
%
%Further insight into this situation was provided by
%Burris and Sankap\-panavar
%\cite{burr:univ81} through consideration of \textit{universal
%algebra}. The main contributions of the paper are as follows.
%
%\begin{itemize}
%\item
%Let $\ph$ be an operation assigning to each structure $\s$ a subalgebra $\ph\s$
%of $\cm\s$. Conditions on $\ph$ are given which guarantee that $\{\s:\ph\s\in
%V\}$ is an elementary class whenever $V$ is a variety, or more generally a
%universal class, of BAO's. These conditions apply when
%$\ph=\sn$, and lead to a structural proof that $\ws V$ is elementary (see \ref{3.2}, \ref{3.3}).
%\item
%Any canonical singleton-persistent variety is elementarily generated (\ref{3.4}).
%\end{itemize}
%
%The  main result is proven in Section \ref{sec:main}.

The theory of substructural logics has developed abruptly in the last years, as it provides a common framework to treat many different logical systems (\cite{GJKO}). Parallel to this growth, the study of residuated lattices has also evolved, since these structures serve as basis for most of the algebraic semantics of substructural logics. In particular, bounded commutative residuated lattices correspond to the Gentzen systems with exchange and weakening (${\bf FL_{ew}}$). The main result of the present work is a categorical equivalence for a subcategory of residuated lattices: Stonean residuated lattices. Stonean residuated lattices can be characterized as the greatest subclass of bounded residuated lattices satisfying that the double negation is a retract onto the Boolean skeleton. They form a variety that we will denote by $\mathbb{SRL}$, which contains among its most important subvarieties the variety of  Boolean algebras, G\"{o}del algebras, product algebras and pseudocomplemented MTL-algebras. Elements of $\mathbb{SRL}$ are closely related to Stone algebras (also known as Stone lattices, see \cite{Gra78}), since the bounded lattice reduct ${\bf L}({\bf A})$ of a distributive Stonean residuated lattice ${\bf A}\in \mathbb{SRL}$ is a Stone algebra.

The motivations for the equivalence that we present are two: on one hand, in 2015, Montagna and Ugolini presented in \cite{MontUgo} a categorical equivalence between product algebras and a category of triples formed by a Boolean algebra, a cancellative hoop and a operator connecting the two structures. Product algebras form a subvariety of pseudocomplemented residuated lattices and the double negation is a retraction onto their Boolean skeletons, therefore they form a subvariety of $\mathbb{SRL}$. On the other hand, in the last century several papers (\cite{CheGra1}, \cite{Kat} and \cite{MadRam}) studied the connection of Stone algebras and another category of triples whose objects are formed by a Boolean algebra, a distributive lattice and a connecting map. These results were generalized in \cite{KatMed} and \cite{Sch1}. In both cases the key fact to obtain the equivalences is the existence of a Boolean retraction, but the third components of the triples (the connecting maps) are quite different. For the case of \cite{MontUgo}, this map must be characterized satisfying some forced properties, while in the other case the map is simply a bounded lattice morphism.

 Although it seems natural to extend both categorical equivalences for the case of  $\mathbb{SRL}$, the triple construction of \cite{CheGra1,Kat} is more appropriate. Therefore our work consists on generalizing that result, and afterwards comparing it with the one of \cite{MontUgo}. This task is far from straightforward, since difficulties appear from the extra structure and new tools to solve the problem are needed. Given an algebra ${\A}\in \mathbb{SRL}$ we consider the triple formed by its Boolean skeleton $\B(\A)$, its algebra of dense elements $\D(\A)$ and a bounded lattice morphism $\phi_A$ from $\B(\A)$ into the lattice of filters of $\D(\A)$.  We define the category  $\sh T$ whose objects are these triples and suitably defined morphisms. Then a functor $\T$ from the category of  $\mathbb{SRL}$ into the category  of triples is introduced. This functor is immediately seen to be faithful and we prove that it is full. To show that it is essentially surjective we adapt the sheaf construction presented in \cite{MadRam} for Stone lattices. Finally, we present some applications of the equivalence.

\section{Background} %

\subsection{Residuated lattices}

 An \textit{integral residuated lattice-ordered  commutative monoid},
 or \textit{residuated lattice} for short, is an algebra $\mbf{A} =
 \langle A, \ast, \to, \lor, \land,  \top\rangle $
  of type $\langle
 2,2,2,2,0\rangle$
  such that
 $\langle A; \ast, \top\rangle$ is a commutative monoid,
 $\mbf{L}(\mbf{A}) = \langle A; \lor, \land, \top\rangle$ is a
 lattice with greatest element $\top,$ and  the following residuation
 condition holds: \bg{equation}\label{resid} x \ast y \leq z, \,\,\,
 \mbox{iff} \,\,\, x \leq y \to z \e{equation} where $x, y, z$ denote
 arbitrary elements of $A$  and $\leq$ is the order given by the
 lattice structure, which is called \textit{the natural order of
 $\A$}.  It is well known that  residuated lattices form a variety, that we shall denote $\bb{RL}$. We will use the same name for the algebraic category of residuated lattices.

\medskip

We list for further reference, some well known
 consequences of (\ref{resid}) that will be used through this
 paper.
  \bg{lemma}\label{basic}
  The following properties hold true
  in any residuated lattice $\A$, where $x,y,z$ denote arbitrary
  elements of $A$:
  \bg{enumerate}[$(i)$]
  \item\label{order} $x \leq y $
  if and only if $x \to y = \top$,
  \item\label{timx} $\top \to x =
  x$,
\item\label{B} $(x \to y) \to ((y \to z) \to (x \to z)) = \top$,
 \item\label{K} $x \to (y \to x) = \top$,
  \item\label{intercambio} $(x \ast y) \to z= x \to (y \to
  z)$,

  \item \label{DeMorgan} $(x \lor y) \to z=(x \to z) \land (y \to z))$.
  \item \label{impmonotder} $x \to (y \land z) = (x \to y) \land (x \to z)$,
  \item \label{distprod} $x \ast (y \lor z) = (x \ast y) \lor (x \ast z)$.
  \e{enumerate}\e{lemma}

 By an {\em implicative filter\/} or {\em i-filter} of a residuated
 lattice $\A$ we mean a subset $F \subseteq A$ satisfying that $\top\in F$ and if $x,x \to y$ are in $F$, then $y \in F.$
 Each i-filter $F$  is the universe of a
 subalgebra of $\A$, which we shall denote $\mbf{F}$, and for each $X\subseteq A$ we denote by $\langle X\rangle$ the i-filter generated by $X$. i.e., the intersection of all i-filters that contain $X$.

  The set of i-filters of $\A$, ordered by inclusion, becomes a bounded lattice, that will be denoted by $\cc{F}_i(\A)$. In this lattice $\{\top\}$ is the bottom, $A$ is the top, and for i-filters $F_1$ and $F_2$, $F_1 \land F_2 = F_1 \cap F_2$ and $F_1 \lor F_2 = \langle F_1 \cup F_2 \rangle$.

 Given an i-filter $F$ of a residuated lattice  $\A$,  the binary
  relation
 $$\theta(F) := \{(x,y) \in A \times A : x \leftrightarrow y \in F\},$$ where $x \leftrightarrow y = (x \to y) \ast (y \to x)$,
  is a congruence on $\A$ such that $F=\top/\theta(F)$, the equivalence class of $\top$. As a matter of fact, the
  correspondence $F \mapsto \theta (F)$ is an  isomorphism
  from $\cc{F}_i(\A)$  onto the  the lattice of congruences of
  $\A$,  whose inverse is given by the map $\theta\mapsto \top/\theta$.

   We will write simply $\A/F$
  instead of $\A/\theta(F)$, and $a/F$  instead of $a/\theta(F)$,   the equivalence
  class determined by $a\in A$.

The variety $\mathbb{RL}$ is \textit{arithmetical}, i.e, it is both congruence-distributive and congruence-permutable (see \cite[Page 94]{GJKO}, \cite[Chapter II \S12]{BurSan}).
 Therefore we have the following form of the Chinese Remainder Theorem (see \cite[Chapter 5, Ex. 68]{Gra68}, \cite[Chapter 4, Ex. 10]{MMT}):
  \bg{theorem}\label{teo:crt} Let $\A \in \bb{RL}$. Given elements $a_1, \ldots, a_n$ in $A$ and i-filters $F_1, \ldots, F_n$ satisfying $a_i/(F_i\lor F_j) = a_j/(F_i \lor F_j)$ for $1 \leq i, j \leq n$, there exists $a \in A$ such that $a/F_i = a_i/F_i$ for all $1 \leq i \leq n$.\e{theorem}

   Since $x \ast y \leq x \land y$, it follows that i-filters are also lattice filters of $\mathbf{L}(\A)$, but the converse is not true:  For each $x \in A$, $$[x) = \{y \in A : x \leq y\}$$ is the lattice-filter generated by $x$, $[x) \subseteq \langle \{x\} \rangle$ but the equality holds if and only if $x^{2} = x$.

   We shall denote by $\cc{F}_l(\A)$ the bounded lattice of (lattice) filters of $\mbf{L}(\A)$.

   \begin{remark}\label{lattice_filter_dist}
   It is well known that $\cc{F}_l(\A)$ is distributive if and only if $\mbf{L}(\A)$ is distributive. In contrast, $\cc{F}_i(\A)$, being isomorphic to the congruence lattice of an algebra in an arithmetical variety, is always distributive.
\end{remark}

 \subsection{Bounded Residuated lattices}

 A \textit{bounded  residuated lattice} is an algebra $\mbf{A} =
 \langle A, \ast, \to, \lor, \land, \top, \bot\rangle$  such that
 $\langle A, \ast, \to,
 \lor, \land, \top\rangle$ is a residuated lattice, and
 $\bot$ is the smallest element of the lattice $\mbf{L}(\mbf{A})$.
 Bounded residuated lattices form a variety, that we shall denote by $\bb{BRL}$.
 For each $\mbf A\in \bb{BRL}$ and $x\in A$ we consider the  unary operation $\neg x =: x \to
 \bot. $

 In the next lemma we collect, for further reference, some properties of bounded residuated lattices.

 \bg{lemma}\label{basicbounded} (\cite[Lemma 1.2]{CigTo12}) The following identities and quasi-identities hold true
  in any residuated lattice $\mbf A$:
  \bg{enumerate}[$a)$]
 \item\label{bb1}  $x \leq y\Ra\neg y \leq \neg x$.
 \item\label{bb3} $x \le \neg \neg x   $.
 \item\label{bb2} $\neg x = \neg \neg \neg x$.
 \item\label{bb4} $x \ra \neg y = y \ra \neg x$.
 \item\label{bb5} $x \ra \neg y = \neg \neg x \ra \neg y$.
 \item\label{bb6} $\neg \neg(x \ra  \neg y) = x \ra  \neg y$.
 \item\label{bb7} $\neg (x \lor y) = \neg x \land \neg y$.
 \end{enumerate}
 \end{lemma}

Let $\A \in \bb{BRL}$. Recall that $x \in A$ is called
  \textit{dense} if $\neg\ x = \bot$. We shall denote by
  $D (\A)$ the set of dense elements of $\A$. In symbols,
  \bg{equation}\label{dense}D (\A) = \{x \in A : \neg\ x = \bot\} = \{x \in A :
  \neg\neg\ x = \top\}.\e{equation}
  It is well known, and easy to prove, that $D(\A)$ is a
  proper i-filter of $\A$. Hence it is the universe of a residuated lattice $\D(\A)$.

\subsection{Boolean elements}

 In general, an element of a bounded non-distributive lattice may have more than one complement. But for $\A\in \mathbb{BRL}$ and $a\in A$, if a complement of $a$ exists then it is unique and is given by $z = \neg a$ (see \cite{KO}). Complemented elements of $\A$ are called \emph{Booleans}. That is, $a \in A$ is Boolean if and only if $a \lor \neg a = \top$ and $a \land \neg a = \bot$. We denote by $B(\A)$ the set of all Boolean elements of $\A$.

 In the next lemma we collect some properties of Boolean elements that we shall use in the sequel.

\bg{lemma}\label{basicboolean} The following properties are true in any bounded residuated lattice $\A$:
 \be[$1)$]
 \item\label{bool0} If $a \in B(A)$, then $\neg a \in B(\A)$ and $\neg\neg a = a$.
 \item\label{bool1} An element $a \in A$ is Boolean if and only if $a \lor\ \neg a = \top$.
 \ee
 If $a, b \in B(\A)$, then for arbitrary elements  $x, y \in A$, one has:
\be[$a)$]
 \item\label{bool2}  $a \ast x = a \land x$,
 \item\label{bool3}  $a \to x = \neg a \lor x$,
 \item \label{boolx} $x = (x \land a) \lor (x \land \neg a)$,
 \item \label{boolxx} If $x \land y \geq \neg b$ and $b \land x = b \land y$, then $x = y$.
 \ee \e{lemma}
 \bg{proof} $\ref{bool0})$ follows from the fact that  $a$ is a complement of $\neg a$.  Since $\neg a$ is the only possible complement of $a$, we have the ``only if" part of $\ref{bool1})$.  To prove the ``if" part, note that by $\ref{bb7})$ in Lemma~\ref{basicbounded}, $a \lor\ \neg a = \top$ implies $\bot = \neg a \land \neg\neg a\ \geq \neg a \land a$. To prove $\ref{bool2})$, suppose first that $y \in A$ and $y \leq a$. Then, taking into account $(\ref{distprod})$ in Lemma~\ref{basic}, $$y = y \ast (a \lor \neg a) = (y \ast a) \lor (y \ast \neg a) = y \ast a.$$ Hence if $y \leq a$ and  $y \leq x$, then $y = y \ast a \leq x \ast a$. Since $x \ast a \leq x, a$, we  proved $\ref{bool2})$. For $\ref{bool3})$, note that by $(\ref{order})$, $(\ref{K})$ and $(\ref{impmonotder})$ in Lemma~\ref{basic}, $x \leq a \to x$ and $\neg a \leq a \to x$. Hence $\neg a \lor x \leq a \to x$. Similarly we have that $$\neg a \leq (a \to x) \to \neg a \leq (a \to x) \to (\neg a \lor x)$$ and taking into account $(\ref{intercambio})$ in Lemma~\ref{basic}, $$a \leq (a \to x) \to x \leq (a \to x) \to (\neg a \lor x).$$ Therefore $\top = a \lor \neg a \leq (a \to x) \to (\neg a \lor x),$ and this implies that we also have $a \to x \leq \neg a \lor x$,   completing the proof of $\ref{bool3})$.
To prove $\ref{boolx})$, note that $x = x \ast (a \lor \neg a) = (x \ast a) \lor (x \ast \neg a) = (x \land a) \lor (x \land \neg a)$.
To prove $\ref{boolxx})$, suppose that $x \land y \geq \neg b$ and $x \land b = y \land b$. Hence by item (\ref{boolx}) we have $x = (x \land b) \lor  \neg b = (y \land b) \lor \neg b = y$.
 %\cqd
 \e{proof}

The following results can be found in \cite{KO}.

\bg{lemma}\label{boolsubalg} For each $\A \in \bb{BRL}$, $B(\A)$ is the universe of a subalgebra $\B(\A)$ of $\A$ which is a Boolean algebra.\e{lemma}

Recall that  an algebra $\A$ is called \textit{directly
 indecomposable} provided that  $A$\ has more than one element  and whenever
 $\A$ is isomorphic  to a direct product of two algebras $\A_1$\ and
 $\A_2$, then either  $\A_1$\   or $\A_2$\ is the trivial algebra
 with just one element.

 \bg{theorem}\label{indecomposable} $\A \in \bb{BRL}$ is directly indecomposable if and only if $\B(\A)$ is the two-element Boolean algebra.\e{theorem}

\section{Stonean residuated lattices}

\subsection{Definition and directly indecomposable Stonean residuated lattices}

\emph{Stonean residuated lattices} are bounded residuated lattices satisfying \bg{equation}\label{eq:stone} \neg x \lor \neg\neg x = \top.\e{equation}
 Let $\bb{SRL}$ be the variety and the category of Stonean residuated lattices and homomorphisms.
\medskip

\bg{remark}\label{nondist} The bounded lattice reduct $\mathbf{L}(\A)$ of an $\A \in \bb{SRL}$ is  a Stone algebra as defined in the lattice literature \cite{Gra78} if and only if $\mathbf{L}(\A)$ is distributive. As noted in \cite{CigEst},  Stonean residuated lattices do not need to satisfy any lattice equation, like distributivity or modularity.\e{remark}

The following theorem is proved in \cite[Theorem 1.7]{CigEst}:
\bg{theorem}\label{equivstonean} The following are equivalent
conditions for a bounded (integral, commutative) residuated lattice
$\A$: \bg{enumerate}[(i)]
\item $\A$ is Stonean,
\item $\A$ satisfies the equations $\neg x\wedge x=\bot
$ and $\neg\ (x \land y) = \neg\ x \lor \neg\
y,$
\item $B(\A) \supseteq \neg (A) := \{\neg\ x : x \in A\}$. %\cqd
\e{enumerate}\e{theorem}

Given  $\A \in \bb{RL}$ and an
element $o \not \in A$, a bounded residuated lattice can be defined with universe $S(\A) = \{o\} \cup A$ adjoining the element $o$ as bottom element of $S(\A)$ and respecting the existing operations on $\A$. Thus $x \ast o = o$ and $\to$ is given by:

$x \to y = \bg{cases} x \to y & \mathrm{if} \,\,\, x, y \in A,\\
o & \mathrm{if} \,\,\, x \in A \,\,\, \mathrm{and}\,\,\, y = o,\\
\top & \mathrm{if} \,\,\, x = o. \e{cases}$

\medskip

The proof of the next theorem consists of some straight computations
that are left to the reader (cf. \cite[pag. 69]{CigTor02}, \cite[Theorem
2.2]{CigTor06}).

\bg{theorem}\label{const} Let $\A \in \bb{RL}$ and $o \not \in A$. Then $$\mbf{S}(\A) = (S(\A), \wedge, \vee, \ast, \to, \top, o)$$ is a
Stonean residuated lattice, and $D(\mbf{S}(\A)) = A$. Moreover, each
homomorphism $h$ from $\A$ into a residuated lattice $\mbf C$ can be
extended to a homomorphism $\mbf{S}(h)$ from $\mbf{S}(\A)$ into
$\mbf{S}(\mbf{C})$ by the prescription $$\mbf{S}(h)(x) = \bg{cases}
h(x) & \mathrm{if}\,\,\, x \in A,\\ o_{\mbf{S}(\mathbf{C})} &
\mathrm{if} \,\,\, x = o_{\mbf{S}(\A)}.\e{cases}$$%\cqd
\e{theorem}

From Theorems \ref{const} and \ref{indecomposable} we obtain:

\bg{lemma}\label{embedding} For every Stonean residuated lattice
$\A$, $\{\bot\} \cup D(\A)$ is the universe of a subalgebra of $\A$
which is isomorphic to  $\mbf{S}(\mbf{D}(\A))$.  $\A$ is isomorphic
to $\mbf{S}(\mbf{D}(\A))$ if and only if $\A$ is directly
indecomposable.\e{lemma}

It is easy to check that $\mbf S$ is a full and faithful functor from the category
$\mathbb{RL}$ of residuated lattices and homomorphisms to the full
subcategory $\mathbb{SRL}_i$  of the category of $\mathbb{SRL}$ and
bottom-preserving  homomorphisms whose objects are the directly indecomposable members of $\bb{SRL}$. Since each object in $\mathbb{SRL}_i$ is of the form $\mbf{S}(\D)$ for some $\D \in \bb{RL}$, by \cite[Proposition 1.3]{Jac} we have:
\bg{corollary}\label{cor:equivcat} The functor $\mbf{S}$ establishes an equivalence between the categories $\mathbb{RL}$ and $\mathbb{SRL}_i$.\e{corollary}

\bg{remark}\label{remark:squotient} {\rm It is easy to check that} for each $\A \in \bb{RL}$ and $F \in \mathcal{F}_i(\A)$,  $\mbf{S}(\A/F) \cong \mathbf{S}(\A)/F$.\e{remark}

We present a weaker form of distributivity that holds in  $\mathbb{SRL}$ that shall be needed later.

\begin{lemma}\label{dist_booleanos}
Let $\A\in \mathbb{SRL}.$ If $x,y,z\in A$, then
\begin{equation}\label{eq:dist_booleanos}
y\vee (\neg\neg z\wedge x)=(y\vee \neg\neg z)\wedge (y\vee x).
\end{equation}
\end{lemma}

\begin{proof} It is enough to check the equation in directly indecomposable algebras. So if ${\bf A}$ is directly indecomposable, $\neg\neg z$ is boolean, therefore $\neg\neg z\in \{\bot, \top\}.$ If $\neg\neg z=\bot$ the lefthandside of (\ref{eq:dist_booleanos}) is $y\vee (\bot\wedge x)=y$ and the righthandside is $(y\vee \bot)\wedge (y\vee x)= y\wedge (y\vee x)=y.$ If $\neg\neg z=\top$ the lefthandside of (\ref{eq:dist_booleanos}) is $y\vee (\top\wedge x)=y\vee x$ and the righthandside is $(y\vee \top)\wedge (y\vee x)= y\vee x.$ Thus the equation holds in the whole variety $\mathbb{SRL}.$
\end{proof}

\subsection{Construction of Stonean residuated lattices as Sheaves}

By a \textit{Boolean space} we mean a totally disconnected compact Hausdorff space, and for each Boolean space $X$, $\C(X)$ will denote the Boolean algebra formed by the clopen subsets of $X$. Moreover, for each $x \in X$, we denote by  $C(x)$ the set of clopen neighborhoods of $x$, i.e.,  $$C(x) = \{a \in C(X) : x \in a\}.$$

Through this Section $\A$ is a fixed non trivial algebra in $\bb{RL}$,  $X$ is a Boolean space and  \fun{\varphi}{\C(X)}{\cc{F}_i(\A)}\ is  a dual homomorphism, i.~e., for $a, b \in C(X)$, $\varphi (a \cap b) = \varphi (a) \lor \varphi (b)$, $\varphi (a \cup b) = \varphi (a) \cap \varphi (b)$,  $\varphi (\emptyset) = A$ and  $\varphi (X) = \{\top\}$.
For $a, b \in C(X)$ such that $a \subseteq b$, let $\rho_{ab}$ be the natural homomorphism from $\A/\varphi (b)$ onto $\A/\varphi (a)$. Then it is easy to check that the system
\bg{equation}
\label{eq:presheaf}\langle\{\A/\varphi(a)\}_{a \in \C(X)}, \{\rho_{ab}\}_{a \subseteq b}\rangle \e{equation} is a presheaf of algebras in $\bb{RL}$.
Since for $a, b \in C(X)$, $\varphi(a) \lor \varphi (b) = \varphi (a \cap b)$, we have that for each $x \in X$,  $$F_x = \bigvee_{a \in  C(x)} \varphi (a)$$  is an i-filter of $\A$.

If for each $a \in C(x)$,  $\rho_{x,a}$ denotes the natural homomorphism from $\A/\varphi(a)$ onto$\A/ F_x$, then the following diagram is commutative for $x \in X$, $a \in C(x)$ and $a \subseteq b$:
\bg{equation}\label{triang}\bg{array}{ccc}
A/\varphi(b) & \stackrel{\rho_{ab}}\longrightarrow & A/\varphi(a) \\
\rho_{x,b} \searrow   & &
\swarrow \, \rho_{x,a}  \\
& A/F_x  & \e{array}\e{equation}

\begin{lemma}
 $A/F_x$
is  the inductive limit of the system $$\langle \{A/\varphi(a)\}_{a \in C(x)}, \{\rho_{a b}\}_{a\subseteq b}\rangle.$$
Moreover, for each $a \in C(X)$ we have:
\bg{equation}\label{eq:intersectionFxina} \bigcap_{x \in a}F_x = \varphi(a).\e{equation}
\end{lemma}

\noindent\bg{proof} That $A/F_x$
is  the inductive limit of the system $$\langle \{A/\varphi(a)\}_{a \in C(x)}, \{\rho_{a b}\}_{a\subseteq b}\rangle$$ is an immediate consequence of the definition of $F_x.$
Now assume that $d \in \bigcap_{x \in a} F_x$. Then for each $x \in X$ there is an $a_x \in C(x)$ such that $d \in \varphi (a_x)$. By compactness, there are $x_1, \dots, x_n \in a$ such that $a = a_{x_{1}} \cup \dots \cup a_{x_{n}}$. Hence $$d \in \bigcap_{i = 1}^n \varphi (a_{x_{i}}) = \varphi (a_{x_{1}} \cup \dots \cup a_{x_{n}}) = \varphi (a).$$
%\cqd

\e{proof}

As a particular case we get that
\bg{equation}\label{eq:intersectionFx} \bigcap_{x \in X}F_x = \{\top\}.\e{equation}
Note that by (\ref{eq:intersectionFx}), $\A$ is a subdirect product of the family $\{A/F_x\}_{x \in X}$.
From Corollary~\ref{cor:equivcat} we get:
\begin{theorem}\label{teo:presheafofSRL} The system
$$\langle\{\Sa(\A/\varphi(a))\}_{a \in \C(X)}, \{\Sa(\rho_{ab})\}_{a \subseteq b}\rangle$$ is a presheaf of directly indecomposable Stonean residuated lattices  and for each $x \in X$, $\Sa(A/F_x)$
is  the inductive limit of the system $$\langle \{\Sa(\A/\varphi(a))\}_{a \in C(x)}, \{\Sa(\rho_{a b})\}_{a\subseteq b}\rangle.$$
\end{theorem}

\medskip

Now let $$\sh{S} = \bigcup_{x \in X} (\{x\} \times \Sa(A/F_x)),$$  and for each $d \in A$, define $$\hat{d}:{X}\to {\sh{S}}$$ by the prescription $\hat{d} (x) = \langle x, d/F_x\rangle$ for each $x \in X$. It follows from (\ref{eq:intersectionFx}) that the correspondence $d \mapsto \hat{d}$ is injective. Moreover, for each $a \in \C(X)$, let $$\hat{a}:{X}\to{\sh{S}}$$ be defined, for each $x \in X$, by the prescription:
\bg{equation} \hat{a}(x) = \bg{cases}
\langle x, o_{\Sa(\A/F_{x})}\rangle & \mathrm{if}\,\,\, x \in a,\\ \langle x,\top \rangle &
\mathrm{if} \,\,\, x \in X \setminus a.\e{cases}\e{equation}
 Observe that for each $x \in a$ we have:
\bg{equation}\label{eq:restriction} \hat{d} (x) = \langle x, \Sa(\rho_{x, a})(d/\varphi_a)\rangle, \,\,\, \hat{a}(x) = \langle x, \Sa(\rho_{x, a})(o_{\Sa(\A/\varphi(a))})\rangle. \e{equation}
Consequently, equipping  $\sh S$ with the topology having as basis the sets $$\hat{d}[a] = \{\hat{d}(x) : x \in a\}\,\,\, \mbox{and} \,\,\, \hat{a}[a] = \{ \hat{a}(x) : x \in a\}$$ for all $a \in C(X)$ and $d \in A$, and defining \fun{\pi}{\sh{S}}{X}\ by the prescription $\pi (\langle x, s\rangle) = x$ for all $\langle x, s\rangle \in \sh{S}$, we have that:

 \begin{theorem}\label{teo:algebra_of_global_sections} (see \cite{Bre,Dav,Kno}) $\langle \sh{S}, \pi, X \rangle$ is the sheaf of directly indecomposable Stonean residuated lattices associated with the presheaf $$\langle\{\Sa(\A/\varphi(a))\}_{a \in \C(X)}, \{\Sa(\rho_{ab})\}_{a \subseteq b}\rangle.$$  The  continuous global sections of $\langle \sh{S}, \pi, X \rangle$, with the operations defined pointwise, form a Stonean residuated lattice that we shall denote by $\cc{A}(\langle \C, \A, \varphi\rangle)$.\end{theorem}

 \medskip

%\begin{remark}\label{re:T_denso} Observe that if the residuated lattice $\A$ in the previous theorem is distributive, the resulting algebra $\A=\cc{A}(\langle \C,\A, \varphi\rangle)$ is a distributive Stonean residuated lattice.
%\end{remark}

 The algebra $\cc{A}(\langle \C, \A, \varphi\rangle)$ will be simply called $\cc A$ when there is no danger of confusion.
The bottom element of $\cc A$ is the global section $O$ defined by $O(x) =  o_ {\Sa(\A/F_{x})}$ for all $x \in X$.
Clearly, the correspondence $a \mapsto \hat{a}$ is a Boolean algebra  isomorphism from $\C (X)$ onto $\B(\cc{A})$. Note that $\hat {X} = O$.

Observe that a section $f$ is dense in $\cc A$ if and only if $f(x) > O(x)$ for all $x \in X$. Hence, taking into account (\ref{eq:intersectionFx}), we have that the mapping $d \mapsto \hat{d}$ is an embedding of $\A$ into $\D (\cc {A})$.

\bg{lemma}\label{onto} Let $Y$ be a closed subset of $X$. If $f$ is a continuous section such that $f(y) > O(y)$ for all $y \in Y$, then there is $d \in A$ such that $f(y) = \hat{d}(y)$ for all $y \in Y$.\e{lemma}
\bg{proof} For each $y \in Y$ there is $d_y \in A$ such that $f(y) = \hat{d_y}(y)$. Since $f, \hat{d_y}$ are continuous sections, there is $a_y \in C(X)$ such that $y \in a_y$ and $f(x) = \hat{d_y}(x)$ for all $x \in a_y$. By compactness, the are  $a_1, \dots, a_n$ in $C(X)$ and $d_1, \ldots, d_n$ in $A$ such that $Y \subseteq a_1 \cup \cdots \cup a_n$ and $f(x) = \hat{d_i}(x)$ for $1 \leq i \leq n$. Since $\hat{d_i}(x) = \hat{d_j}(x)$ for all $x \in a_i \cap a_j$, then   $d_i/F_x = d_j/F_x$ for $x \in a_i \cap a_j$, $1 \leq i, j \leq n$. By (\ref{eq:intersectionFxina}) $d_i \leftrightarrow d_j \in \varphi (a_i \cap a_j) = \varphi (a_i) \lor \varphi (a_j)$. Hence by  Theorem~\ref{teo:crt} we have that there is $d \in A$ such that $\hat{d} (x) = \hat{d_i} (x)$ for all $x \in a_i$, and this implies that $f(y) = \hat{d}(y)$ for all $y \in Y$.%\cqd

\e{proof}

\begin{lemma} The residuated lattices $\A$ and $\D(\cc{A})$ are isomorphic.
\end{lemma}
\bg{proof}
We already noted that the mapping $d \mapsto \hat{d}$ is an injective homomorphism from $\A$ into $\D(\cc{A})$. Taking $Y = X$ in the above lemma, we see that it is also surjective. Hence it is a residuated lattice isomorphism from $\A$ onto $\D(\cc {A})$.%\cqd
\e{proof}
\medskip

For each $a \in C(X)$, let $\mathrm{Sec}(a)$ denote the Stonean residuated lattice of restrictions of $\cc A$ to $a$.

\begin{lemma}
The mapping $d/\varphi(a) \mapsto \hat{d}\!\!\restriction_{a}$  is an isomorphism from $\A/\varphi(a)$ onto $\D(\mathrm{Sec}(a))$.
\end{lemma}

\begin{proof}
It follows from (\ref{eq:intersectionFxina}) that if $d_1/\varphi(a) = d_2/\varphi(a)$, then $\hat{d_1}(x) = \hat{d_2}(x)$ for all $x \in a$. Hence the mapping $d/\varphi(a) \mapsto \hat{d}\!\!\restriction_{a}$ is an injective homomorphism  from $\A/\varphi(a)$ into $\D(\mathrm{Sec}(a))$, and it follows from Lemma~\ref{onto} that this homomorphism is also surjective.
\end{proof}

%%%%%%%%%%%%%%%%%%%%%%%%%%%%%%%%%%%%%%%%%%%%%%%%%%%%%%%%%%%%%%%%%%%%%%%%%%%%%%%%%%%%%%%%%%%%%%%%%%%%%%%%%%%%%%%%%%%%%%%%%%%%%%%%%%%%%%%%%%%%%%%%%%%%%%%%%%%%%

\section{The triple associated with a Stonean residuated lattice}

\subsection{Representation of elements and triples}

Let $\A \in \bb{BRL}$. A \textit{Boolean retraction} is a homomorphism \fun{h}{\A}{\B(\A)}\ such that $h(h(x)) = h(x)$ for all $x \in A$. It is well known that $\A \in \bb{SRL}$ if and only if the double negation $\dn$ is a Boolean retraction. The next lemma shows that it is the only possible Boolean retraction $h$ satisfying $x \leq h(x)$ for all $x \in A$.

\bg{lemma} Let $\A \in \bb{BRL}$ and let \fun{h}{\A}{\B(\A)}\ be a Boolean retraction. If $x \leq h(x)$ for all $x \in A$, then $\dn x = h(x)$ for all $x \in A$.\e{lemma}
\bg{proof} Since $x \leq h(x)$, one has that $\top = \neg h(x) \lor h(x) \leq \neg x \lor h(x)$. On the other hand $\neg x \land h(x) \leq h(\neg x \land h(x)) = \neg h(x) \land h(x) = \bot.$
Therefore $\neg x = \neg h(x)$ and $\dn x = \dn h(x) = h(x)$. %\cqd
\e{proof}

\bg{lemma}\label{stonedecomposition} Let $\A \in \bb{SRL}$. Then for every $x \in A$:
\bg{equation}\label{eq:stonedecomposition} x = \neg\neg x \ast (\neg\neg x \to x).\e{equation}\e{lemma}
\bg{proof} Since $\neg x \in B(\A)$ for each $x \in A$,  by Lemma~\ref{basicboolean} we have:
$$\neg\neg x \ast (\neg\neg x \to x) = \neg\neg x \land (\neg x \lor x) =  \neg\neg x \land x = x.$$\e{proof}%\texttt{\cqd}

Since $\neg\neg (\neg\neg x \to x) = \neg \neg (\neg x \lor x) = \neg (\neg\neg x \land \neg x) = \top,$
 (\ref{eq:stonedecomposition}) says that each element $x$ of $\A  \in \bb{SRL}$ can be written as \bg{equation}\label{eq:stonedecomposition2}x = b \ast d = b \land d,\e{equation} where $b \in B(\A)$ and $d \in D (\A), d \geq \neg x$. Hence $\A = B(\A) \ast D(\A)$.

 Observe that $b$ in (\ref{eq:stonedecomposition2}) has to be $\neg\neg x$. Indeed,
 \begin{equation}\label{doblenegacion}
 \neg\neg x = \neg\neg(b \ast d)= \neg (b \to \neg d) = \neg\neg b = b.
 \end{equation}
 The dense element $d$ is not uniquely determined by $x$, as the following trivial example shows: suppose that there is $\top > d \in D(\A)$, then $\bot = \bot \ast d = \bot \ast \top$. But it follows from (\ref{boolxx}) in Lemma~\ref{basicboolean} that \textit{$\neg\neg x \to x$ is the only dense element $d$ satisfying (\ref{eq:stonedecomposition}) such that $d \geq \neg x$}.

 \bg{lemma}\label{decompimstone} Let $\A \in \bb{BRL}$. If each $x \in A$ can be written as $x = b \ast d$, with $b \in B(\A)$ and $d \in D(\A)$,  then $\A \in \bb{SRL}$. \e{lemma}
 \bg{proof} For each $x \in A$ one has $\neg x = \neg (b \ast d) = b \to \neg d = \neg b$. Hence $\neg x \in B(\A)$ for all $x \in A$, and the result follows from Theorem~\ref{equivstonean}.
 \e{proof}

Let $\A \in \bb{SRL}$. For each $a \in B(\A)$ let \bg{equation}\label{eq:fi}\phi_{\A} (a) = \{x \in D(\A) : x \geq \neg a\} = [\neg a) \cap D(\A).\e{equation}
It is easy to check that the correspondence $a \mapsto \phi_{\A} (a)$ defines a lattice homomorphism from $\B(\A)$ into the lattice $\mathcal{F}_i(\D(\A))$ of lattice filters of $\D(\A)$.

\bg{defin}\label{triple} The triple associated with $\A \in \bb{SRL}$ is $\langle \B(\A), \D(\A), \phi_{\A}\rangle$.\e{defin}

%Observe that if ${\bf A}$ is distributive the second component of the triple, $\D(\A)$ is also distributive. We will only be concerned with this particular case.

From Lemma \ref{stonedecomposition} and the comments after it we can assert that the correspondence \bg{equation}\label{eq:psi}x \mapsto  (\dn x, \dn x \to x)\e{equation} defines a bijective mapping  from $A$ onto the set $$P(\A) = \{(a,d) \in B(\A) \times D(\A) : d \in \phi_{\A}(a)\}.$$

\begin{example}

We will provide an example of how non-isomorphic Stonean residuated lattices may have the same algebras of Boolean and dense elements. The idea is borrowed from \cite{MontUgo}, but adapted to our notation.

Let ${\mbf B}$ be the two-elements Boolean algebra and ${\mbf C}$ an arbitrary residuated lattice. Consider the algebras $${\mbf A}_1= (\mbf{S}({\mbf C}))^\omega$$ and $${\mbf A}_2= \mbf{S}({\mbf C}^\omega)\times {\mbf B}^\omega.$$

Clearly ${\mbf B}({\mbf A}_1)\cong {\mbf B}({\mbf A}_2)\cong {\mbf B}^\omega$ and ${\mbf D}({\mbf A}_1)\cong {\mbf D}({\mbf A}_2)\cong {\mbf C}^\omega$. Consider the element $b\in{\mbf B}^\omega$ defined by $b_1=\bot$ and $b_n=\top$ for all $n>1.$

If $ \langle{\mbf B}({\mbf A}_1), {\mbf D}({\mbf A}_1), \phi_1\rangle$ and $\langle {\mbf B}({\mbf A}_2), {\mbf D}({\mbf A}_2), \phi_2\rangle$ are the  triples associated to $\A_1$ and to $\A_2$ respectively, then $\phi_1(b)\cong {\mbf C}^\omega$ while $\phi_2(b)$ is the trivial filter whose only element is $\top$.
\end{example}

\subsection{The category of triples}
We define the \textit{category of triples}, that we shall denote by $\sh T$, as follows:

\medskip

\noindent \textit{Objects:} Triples $(\B, \D, \phi)$ such that $\B$ is a Boolean algebra, $\D$ is a residuated lattice and $\phi$ is lattice-homomorphism preserving $\bot$ and $\top$, from $\B$ into $\cc{F}_i(\D)$.

\noindent \textit{Morphisms:} Given triples $(\B_i, \D_i, \phi_i)$, $i = 1, 2$, we say that a morphism from $(\B_1, \D_1, \phi_1)$ to $(\B_2, \D_2, \phi_2)$  is a pair $(h, k)$ such that:
\bg{enumerate}
\item[$M_1$] \fun{h}{\B_1}{\B_2}\ is a Boolean algebra homomorphism,
\item[$M_2$] \fun{k}{\D_1}{\D_2}\ is a residuated lattice homomorphism, and
\item[$M_3$] For all $a \in B_1$, $k(\phi_1(a)) \subseteq \phi_2(h(a))$.
\e{enumerate}

\medskip

It is easy to check that given morphisms
\begin{center}
\fun{(h, k)}{(\B_1, \D_1, \phi_1)}{(\B_2, \D_2, \phi_2)}\
\end{center}
and
\begin{center}
\fun{(h',k')}{(\B_2, \D_2,\phi_2)}{(\B_3, \D_3, \phi_3)}\
\end{center}
the composition  $(h'h,k'k)$ is a morphism from $(\B_1, \D_1, \phi_1)$ to $(\B_3, \D_3, \phi_3)$, and that this composition of morphisms is associative. Moreover, for each object $(\B, \D, \phi)$, the pair $(id_{\B}, id_{\D})$ is a morphism which is a unit for composition. Hence $\sh T$ as defined above is really a category.

Condition $M_3$ is equivalent to the assertion that the correspondence
\bg{equation}\label{eq:morphism-quotients} d/\phi_1(a) \mapsto k (d)/\phi_2(h(a))\e{equation}
is a well defined function from  $ {\mbf D}_1/\phi_1(a)$ onto $k({\mbf D}_1)/\phi_2(h(a))$ for all $a \in B$. Clearly this function is a residuated lattice homomorphism.

\bg{lemma}\label{isomorphismsinT} A morphism \fun{(h,k)}{(\B_1, \D_1, \phi_1)}{(\B_2, \D_2, \phi_2)}\ is an isomorphism in $\sh T$ if and only if the following conditions are satisfied:
\bg{enumerate}
\item[$I_1$]  $h$ is an isomorphism from $\B_1$ onto $\B_2$,
\item[$I_2$] $k$ is an isomorphism from $\D_1$ onto $\D_2$, and
\item[$I_3$] For all $a \in B_1$, $k(\phi_1 (a)) = \phi_2(h(a))$.
\e{enumerate}\e{lemma}
\bg{proof} Suppose that
\begin{center}
\fun{(h,k)}{(\B_1, \D_1, \phi_1)}{(\B_2, \D_2, \phi_2)}\
\end{center}
is an isomorphism. Then there is a morphism
\begin{center}
\fun{(h',k')}{(\B_2, \D_2, \phi_2)}{(\B_1, \D_1, \phi_1)}\
\end{center}
such that $(h'h,k'k) = (id_{\B_{1}}, id_{\D_{1}}) = (hh',kk')$. This implies that conditions $I_1$ and $I_2$ are satisfied and that $h' = h^{-1}$ and $k' = k^{-1}$. Since for each $a \in B_1$, $k'(\phi_2 (h (a))) \subseteq \phi_1(h'(h(a))) = \phi_(a)$, we have $$\phi_2(h(a)) = k(k'(\phi_2(h(a)))) \subseteq k(\phi_1(a)) \subseteq \phi_2 (h(a)),$$ and this proves $I_3$.

\noindent Suppose now that \fun{(h,k)}{(\B_1, \D_1, \phi_1)}{(\B_2, \D_2, \phi_2)}\ satisfy $I_1$, $I_2$ and $I_3$. To complete the proof it is sufficient to show that $$\fun{(h^{-1},k^{-1})}{(\B_2, \D_2, \phi_2)}{(\B_1, \D_1, \phi_1)}$$ is a morphism in $\sh T$. Obviously $(h^{-1}, k^{-1})$ satisfy $M_1$ and $M_2$. On the other hand, $k^{-1}(\phi_2(h(a))) = k^{-1}(k(\phi_1(a))) = \phi_1(h^{-1}(h(a)))$, and since for each $b \in B_2$ there is $a \in B_1$ such that $b = h(a)$, we have that also $M_3$ is satisfied.%\cqd
\e{proof}

\medskip
\subsection{The functor $\T$}

In this section we will define a functor $\T$ from the category $\mathbb{SRL}$ of Stonean residuated lattices into the category  $\sh T$ of triples.

\medskip

For each $\A \in \bb{SRL}$ define
\begin{equation}\label{functor_T}
\T(\A) = (\B(\A), \D(\A), \phi_{\A}).
 \end{equation}
 Suppose that  $\A_1, \A_2 \in \bb{SRL}$ and \fun{f}{\A_1}{\A_2}\ is an homomorphism. For $i=1,2$, let $\B_i=\B(\A_i),$  $\D_i=\D(\A_i)$ and $\phi_{{\mbf A}_i}=\phi_i.$  We have that for each $a \in B_1$, $f(\phi_1 (a))\subseteq \phi_2 (f(a))$, and  since $f(B_1) \subseteq B_2$ and $f(D_1) \subseteq D_2$, the pair of restrictions $(f\!\!\upharpoonright_{\B_{1}}, f\!\!\upharpoonright_{\D_{1}})$ is a morphism from $\T(\A_1)$ to $\T(\A_2)$ in $\sh T$. Hence if we define $$\T(f) =  (f\!\!\upharpoonright_{\B_{1}}, f\!\!\upharpoonright_{\D_{1}}),$$ it follows that $\T$ is  a functor from $\mathbb{SRL}$ to $\sh T$.
\bg{remark}\label{Tfaithful} It follows from Lemma~\ref{stonedecomposition} that the functor $\T$ is faithful: Given homomorphisms $f, g$ from $\A_1$ into $\A_2$, if $\T(f) = \T(g)$, then $f = g$.%\cqd
\e{remark}

The following two results, that follow from $M_3$, will play a crucial role next. For both of them assume that $\A_1, \A_2$ are Stonean residuated lattices and $(h, k)$ will denote a morphism from $\T(\A_1)$ into $\T(\A_2)$.
\bg{lemma}\label{unicidad} Let $a \in B(\A_1)$ and $d, e \in D(\A_1)$. If $a \ast d = a \ast e$, then $h(a) \ast k(d) = h(a) \ast k(e)$.\e{lemma}
\bg{proof}
It is easy to see that for $d, e\in D(\A_1)$ one has that $a\wedge d=a\ast d=a\ast e=a\wedge e$ if and only if $m=(e\to d)\ast(d\to e)\in \phi_1(\neg a).$ Then $k(m)$ is in $\phi_2(h(\neg a))$ and this implies that $h(a)\ast k(d)=h(a)\ast k(e).$%\cqd

\e{proof}

\begin{lemma}\label{dist_sup} Let $a\in B(\A_1)$ and $d\in D(\A_1)$. Then $k(a\vee d)=h(a)\vee k(d)$.\end{lemma}
\begin{proof}
First, as $\neg(\neg a)=a\leq a\vee d$ and $a\vee d\in D(\A_1)$, $a\vee d\in \phi_1(\neg a)$. Therefore $k(a\vee d)\in \phi_2(h(\neg a))$, but this means that $h(a)=\neg h(\neg a)\leq k(a\vee d)$, and since $k$ is order preserving we get $h(a)\vee k(d)\leq k(a\vee d)$.

For the other inequality, observe that $\neg a=a\to\bot\leq a\to d=(a\vee d)\to d$ and that $(a\vee d)\to d$ is dense, therefore $(a\vee d)\to d\in \phi_1(a)$. As before, $k((a\vee d)\to d)\in \phi_2(h(a))$, so $\neg h(a)\leq k((a\vee d)\to d)=k(a\vee d)\to k(d)$, and by residuation this is equivalent to
$$k(a\vee d)\leq \neg h(a)\to k(d)=h(a)\vee k(d).$$%\cqd
\end{proof}

If $x \in A_1$ can be written as $x = \neg\neg x \ast d$, where $d \in D(\A_1)$, we  set
 \bg{equation} \label{eq:def f} f(x) = h(\neg\neg x) \ast k(d),\e{equation} and it follows from the above Lemma \ref{unicidad} that $f$ is a well defined function from $A_1$ into $A_2$.

 \bg{theorem}\label{f_morphism} The function \fun{f}{A_1}{A_2}\ defined by (\ref{eq:def f}) satisfies the following properties:
\bg{enumerate}[(i)]
\item $f\!\!\upharpoonright_{B(\A_{1})} = h$, $f\!\!\upharpoonright_{D(A_{1})} = k$,
\item $f(\bot) = \bot$, $f(\top) = \top$,
\item $f(x \land y) = f(x) \land f(y)$,
\item $f(x \ast y) = f(x) \ast f(y)$,
\item $f(x \to y) = f(x) \to f(y)$,
\item $f(x \lor y) = f(x) \lor f(y).$
\end{enumerate}
\e{theorem}

\bg{proof} Properties $(i)$ and  $(ii)$ are rather obvious. For the remainder of the proof we assume that $x = \neg\neg x \land d$ and $y = \neg\neg y \land e$, with $d, e \in D(\A_1)$. $(iii)$ follows from the fact that
$$x \land y = (\neg\neg x \land d) \land (\dn y \land e) = \dn (x\land y) \land (d \land e).$$
The proof of $(iv)$ is similar.

To prove  $(v)$ recall that if $a\in B({\mbf A}_1)$ and $r\in D({\mbf A}_1)$ then
$r\le a\to r$, therefore $a\to r\in D({\mbf A}_1)$. Using $d)$ in Lemma \ref{basicbounded} we also have $r\to a=a$, because
$$r\to a=r\to \neg\neg a=\neg a\to \neg r=\neg a\to \bot=\neg\neg a=a.$$
With the previous in mind, we can prove
\begin{align*}f(x)\to f(y)&=\left(h(\neg\neg x)\wedge k(d)\right) \to \left(h(\neg\neg y)\wedge k(e)\right)\\
&= \left(\left(h(\neg\neg x)\wedge k(d)\right) \to h(\neg\neg y)\right)\wedge \left(\left(h(\neg\neg x)\wedge k(d)\right) \to k(e)\right)\\
&= \left(k(d) \to \left(h(\neg\neg x)\to h(\neg\neg y)\right)\right)\wedge \left(h(\neg\neg x)\to\left(k(d)\to k(e)\right)\right)\\
&= \left(h(\neg\neg x)\to h(\neg\neg y)\right)\wedge \left(h(\neg\neg x)\to k(d\to e)\right)\\
&= h\left(\neg\neg x\to \neg\neg y)\right)\wedge \left(h(\neg\neg x)\to k(d\to e)\right).
\end{align*}

\noindent As $\neg\neg x$ (and therefore $\neg x$) is Boolean and $(d\to e)$ is dense, by Lemma \ref{dist_sup},
\begin{align*}h(\neg\neg x)\to k(d\to e) &= \neg h(\neg\neg x)\vee k(d\to e)\\
&= h(\neg x)\vee k(d\to e)\\
&= k(\neg x \vee (d\to e))\\
&= k(\neg\neg x \to (d\to e)).\end{align*}

\smallskip

Finally, we obtain
\begin{align*}f(x)\to f(y) &=h\left(\neg\neg x\to \neg\neg y\right)\wedge k(\neg\neg x \to (d\to e))\\
&= f\left((\neg\neg x\to \neg\neg y)\wedge (\neg\neg x \to (d\to e))\right)\\
&= f\left((\neg\neg x\to (d\to \neg\neg y))\wedge (\neg\neg x \to (d\to e))\right)\\
&= f\left(((\neg\neg x *d)\to \neg\neg y)\wedge ((\neg\neg x*d) \to e)\right)\\
&= f\left((x\to \neg\neg y)\wedge (x \to e)\right)\\
&= f\left(x\to (\neg\neg y\wedge e)\right)\\
&= f(x\to y).\end{align*}

Using that $f$ preserves $\wedge$, Lemma \ref{dist_booleanos} and Lemma \ref{dist_sup} we get $(vi)$, as
\begin{align*} f(x\vee y) &= f((\neg\neg x \wedge d)\vee(\neg\neg y \wedge e))\\
	&= f((\neg\neg x \vee \neg\neg y)\wedge (\neg\neg x \vee e)\wedge (d\vee \neg\neg y)\wedge(d \vee e))\\
	&= f(\neg\neg x \vee \neg\neg y)\wedge f(\neg\neg x \vee e)\wedge f(d\vee \neg\neg y)\wedge f(d \vee e)\\
	&= h(\neg\neg x \vee \neg\neg y)\wedge k(\neg\neg x \vee e)\wedge k(d\vee \neg\neg y)\wedge k(d \vee e)\\
	&= (h(\neg\neg x) \vee h(\neg\neg y))\wedge (h(\neg\neg x)\vee k(e)) \\
	&\hspace{1cm}\wedge (k(d)\vee h(\neg\neg y))\wedge (k(d) \vee k(e))\\
	&= (h(\neg\neg x) \wedge k(d))\vee(h(\neg\neg y) \wedge k(e))\\
	&= f(x)\vee f(y).\end{align*}
\e{proof}

%Observe that the distributivity of ${\bf A}_1$ and ${\bf A}_2$ is only used to prove item (vi).
Therefore we have proved that, if $\A_1$ and $\A_2$ are in $\mathbb{SRL}$, given a morphism $(h,k)$ from $\T(\A_1)$ to $\T(\A_2)$ there exists a homomorphism $f$ as given in (\ref{eq:def f}) from $\A_1$ to $\A_2$ such that $\T(f)=(h,k).$ Summing up we get:

\begin{corollary}\label{Tfull}  The functor $\T$ is full.
\end{corollary}

To close this section we will use the results about sheaves to prove that the functor $\T$ is dense.

\begin{theorem}\label{teo:T_denso} The functor $\T$ is dense, i.e., for each triple $(\B, \D, \phi)$ there is $\A \in \mathbb{SRL}$ such that $\T(\A)=(\B, \D, \phi).$
\end{theorem}

\noindent\begin{proof}
Given a triple $( \B, \D, \phi )$, take $X$ to be the Stone space of the Boolean algebra $\B$, and  \fun{\varphi}{\C(X)}{\cc{F}_i(\D)}\ given by $\varphi (a) = \phi (\neg a)$ for each $a\in \B$. Then $\varphi$ is a dual homomorphism and, according to Theorem \ref{teo:presheafofSRL} the system  $$\langle\{\Sa(\D/\varphi(a))\}_{a \in \C(X)}, \{\Sa(\rho_{ab})\}_{a \subseteq b}\rangle$$ is a presheaf of directly indecomposable Stonean residuated lattices. Therefore let

 $$\A=\cc{A}(\langle X, \D, \varphi\rangle)$$ be the algebra of continuous global sections of the sheaf $\sh S$ constructed in Theorem \ref{teo:algebra_of_global_sections}, which is in $\mathbb{SRL}.$

Let \fun{h}{\B}{\B({\A})}\ and \fun{k}{\D}{\D(\A)}\ be defined by $h(a) = \hat{a}$ for each $a \in B$ and by $k(d) = \hat{d}$ for all $d \in D$. Clearly $h$ is an isomorphism from $\B$ onto $\B(\A)$ and $k$ is an isomorphism from $\D$ onto $\D(\A)$, and by taking into account  (\ref{eq:intersectionFxina}) and Lemma~\ref{onto} we have the following equivalences, for $a \in B$ and $d \in D$:
$$d \in \phi (a) = \varphi (\neg a) \Leftrightarrow d \in \bigcap_{x \in \neg a} F_x \Leftrightarrow \hat{d} = \top \,\, \mbox{for all}\,\, x \in \neg a \Leftrightarrow \hat{d} \in \phi_{\cc{A}}(\hat{a}).$$
Hence $k(\phi (a)) = \phi_{\cc{A}}(h(a))$, and $(h, k)$ is an isomorphism between $(\B,\D, \phi)$ and $\T(\A)$.
%\cqd
\end{proof}

Following \cite[page 91]{MacLane}, from Remark \ref{Tfaithful}, together with Corollary \ref{Tfull} and Theorem \ref{teo:T_denso} we can assert:
 \begin{theorem}
 The functor $\T$ from $\mathbb{SRL}$ to  $\sh{T}$ that acts on objects as $\T(\A)=(\B(\A), \D(\A), \phi_{\A})$ and if $f:\A_1\to \A_2$ is homomorphism of Stonean residuated lattices $$\T(f) =  (f\!\!\upharpoonright_{\B(\A_1)}, f\!\!\upharpoonright_{\D(\A_1)})$$ defines an equivalence of categories.
 \end{theorem}

\section{Applications}

\subsection{Equations and restrictions of the functor ${\bf T}.$}

\begin{lemma}\label{lemma:subvariedades}(\cite[Th. 2.6]{Cig07}) For each subvariety $\bb{V}$ of $\bb{SRL}$, $$\bb{V}^\ast=\{{\A}\in \bb{RL}: {\mbf S(\A)}\in\bb{V}\}$$ is a subvariety of $\bb{RL}$.\end{lemma}

Taking into account that subdirectly irreducible algebras %in a variety of Stonean residuated lattices
are directly indecomposable, we have:

\begin{corollary}\label{coreqn}Let $\bb{V}$ be a subvariety of $\bb{SRL}$ and $\epsilon$ be an equation in the language of residuated lattices that is satisfied in $\bb{V}^\ast$. Then $\epsilon$ is satisfied in $\bb{V}$ if and only if for all $\A\in \bb{V}^\ast$, $\epsilon$ is satisfied in $\mbf{S}(\A)$.\end{corollary}

For each $\mathbb{V}\subseteq \mathbb{SRL}$ let $\sh{T}_{\mathbb{V}}$ be the full subcategory of $\sh{T}$ whose objects have as second component an element in $\mathbb{V}^{\ast},$ i.e., $(\B,\D, \phi)$ is an object of $\sh{T}_{\mathbb{V}}$ if and only if $\D\in \mathbb{V}^{\ast}.$ Because of the previous results we can conclude:

 \begin{theorem}
The restriction of the functor ${\bf T}$ to $\mathbb{V}$ defines an equivalence between the categories $\mathbb{V}$ and $\sh{T}_{\mathbb{V}}$.
\end{theorem}

\begin{proof}
If ${\bf A}\in \mathbb{V}$, since ${\bf S}({\bf D}({\bf A}))$ is a subalgebra of ${\bf A}$ we get that ${\bf T}({\bf A})$ is an object in $\sh{T}_{\mathbb{V}}.$ Then it is only left to verify that the restriction of ${\bf T}$ to $\mathbb{V}$ is dense on $\sh{T}_{\mathbb{V}}.$ To this aim, consider a triple $({\bf B}, {\bf D}, \phi)\in \sh{T}_{\mathbb{V}}$, i.e., ${\bf D}\in \mathbb{V}^*.$ Following the proof of Theorem \ref{teo:T_denso}, the algebra $\A=\mathcal{A}(\langle {\bf B}, {\bf D}, \varphi\rangle)$ constructed in Theorem \ref{teo:algebra_of_global_sections} is such that ${\bf T}({\bf A})$ is isomorphic to $({\bf B}, {\bf D}, \phi).$ To see that ${\bf A}$ is in $\mathbb{V}$ we just see that ${\bf A}$ is the algebra of continuous global sections of a sheaf associated to the presheaf $\{ {\bf S}({\bf D}/\varphi(a))\}_{a\in {\bf B}} $.  But since each ${\bf D}/\varphi(a)$ is a homomorphic image of a residuated lattice in $\mathbb{V}^*$, then by Lemma \ref{lemma:subvariedades} each directly indecomposable algebra ${\bf S}({\bf D}/\varphi(a))$ is in $\mathbb{V}$ and we get the desired result.
\end{proof}

Observe that if $\mathbb{B}$ is the category of Boolean algebras, then $\mathbb{B}^{\ast}$ has as only element the trivial algebra $\{\top\}.$ Then the triples in the category  $\sh{T}_{\mathbb{B}}$ can be distinguished only by the first component, the Boolean part.

\begin{example} Let $\mathbb{DSRL}$ be the subvariety of $\mathbb{SRL}$ whose elements are distributive Stonean residuated lattices. Thus $\mathbb{DSRL}^{\ast}$ is the variety of distributive residuated lattices and the category $\mathbb{DSRL}$ is equivalent to the category $\sh{T}_{\mathbb{DSRL}}$ whose objects are triples such that the second component is a distributive residuated lattice.
\end{example}

\begin{example}\label{product_example} Consider the subvariety $\mathbb{C}$ of $\mathbb{RL}$ characterized by the equations:
\begin{align*}
x\to x\ast y &= y\\
x\ast(x\to y) &= x\wedge y\\
(x\to y)\vee (y\to x) &=\top.
\end{align*}
That is, $\mathbb{C}$ is the variety of cancellative divisible and prelinear residuated lattices (cancellative basic hoops). The reader can verify that the variety $\mathbb{P}$ of product algebras is such that $\mathbb{P}^{\ast}=\mathbb{C}.$ Then the functor ${\bf T}$ restricted to $\mathbb{P}$ provides a categorical equivalence for product algebras. We will come back to this example in Section \ref{product_algebras}.
\end{example}

To close this section we state the following result:

\begin{lemma}\label{thmeqn}Let $\tau(x_1,\ldots,x_n)=\top$ be an equation in the language of residuated lattices.
A variety $\bb{V}\subset\bb{SRL}$ is such that $\bb{V}^\ast$ satisfies the equation if and only if $\bb{V}$ satisfies
\begin{align}\label{eqdense}\tau(\neg\neg x_1\to x_1,\ldots,\neg\neg x_n\to x_n)=\top.\end{align} \end{lemma}
\begin{proof}Assume that
%$\bb{V}^\ast\subset\bb{W}$
$\bb{V}^\ast$ satisfies $\tau=\top$ and let $\A\in\bb{V}$. Since by Lemma~\ref{embedding}, $\D(\A)$ is in $\bb{V}^\ast$ and for each element $x\in A$ we have $\neg\neg x\to x\in D(\A)$, then  (\ref{eqdense}) is satisfied in $\A$. On the other hand, if each element of $\bb{V}$ satisfies (\ref{eqdense}), then for $\D\in\bb{V}^\ast$ we have that $\mbf{S}(\D)$ satisfies (\ref{eqdense}) and as $\neg\neg d\to d=d$ for each $d\in D$, we get that
%$\D\in\bb{W}$
$\D$ satisfies $\tau=\top$ as desired.\end{proof}

\subsection{Free algebras}

Free algebras in a variety $\bb{V}$ of Stonean residuated lattices are described in \cite{Cig07}. We will recall the description and explicitly show the triple associated to the finitely generated free algebra in $\bb{SRL}$.

For each set of generators $X$ and each variety $\bb{V}$, let $\rm{Free}_\bb{V}(X)$ be the free algebra over $X$ in $\bb{V}$.

By Theorem 3.1 in \cite{Cig07}, for each variety $\bb{V}$ of Stonean residuated lattices, the Boolean skeleton of the free algebra over the set $X$ is completely described as  $$B(\rm{Free}_\bb{V}(X))\cong \rm{Free}_\bb{B}(X)\cong {\mbf B}^{2^X},$$
where $\bb{B}$ is the variety of Boolean algebras and $\mbf B$ is the two element Boolean algebra.

Now, from Theorem 3.6 in \cite{Cig07}, for each variety $\bb{V}$ of Stonean residuated lattices such that $\bb{B}\subsetneq\bb{V}$, the free algebra of Stonean residuated lattices over the set of generators $X$, $\rm{Free}_\bb{V}(X)$, is representable as a weak Boolean product of the family $\left({\mbf S}(\rm{Free}_{\bb{V}^\ast}(Y))\right)_{Y\subset X}$ over the Cantor space ${\mbf 2}^X$ (the dual of $\rm{Free}_\bb{B}(X)$).

Therefore the dense elements $D(\rm{Free}_\bb{V}(X))$ will be a subdirect product of the family $\left(\rm{Free}_{\bb{V}^\ast}(Y)\right)_{Y\subset X}$, and for ${\mbf b}=(b_Y)_{Y\subset X}$ we have $$\phi({\mbf b})=\{{\mbf d}\in D(\rm{Free}_\bb{V}(X)): {\mbf d}\geq \neg{\mbf b}\}.$$

For the finite case, recall that Boolean products coincide with direct products, so from \cite[Co. 3.7]{Cig07} we obtain the following result.

\begin{theorem}For each variety $\bb{V}$ of Stonean residuated lattices such that $\bb{B}\subsetneq\bb{V}$ and each $n\geq 1$,
$$\rm{Free}_\bb{V}(n)\cong\prod_{k=0}^n \left({\mbf S}(\rm{Free}_{\bb{V}^\ast}(k))\right)^{n \choose k}.$$
Therefore its associated triple is isomorphic to $({\mbf B}^{2^n},\D,\phi)$, where $\mbf B$ is the two element Boolean algebra as before,
%$$B(\rm{Free}_\bb{V}(n))\cong{\mbf B}^{2^n},$$
$$\D\cong\prod_{k=0}^n \left(\rm{Free}_{\bb{V}^\ast}(k)\right)^{n \choose k}\cong\prod_{j=1}^{2^n} \rm{Free}_{\bb{V}^\ast}(k_j) ,$$
where we define for $j=1,\ldots,2^n$
$$k_j=\min\left\{k: \sum_{l=0}^k {n\choose l} \geq j\right\}$$
%then $\prod_{k=0}^n \left(\rm{Free}_{\bb{V}^\ast}(k)\right)^{n \choose k}\cong\prod_{j=1}^{2^n} \rm{Free}_{\bb{V}^\ast}(k_j)$
and for ${\mbf b}=(b_1,\ldots,b_{2^n})\in {\mbf B}^{2^n}$ we have that
$$\phi(b_1,\ldots,b_{2^n})=\prod_{j=1}^{2^n} \phi_{j,b_j}$$
where
$$\phi_{j,b_j}=\left\{\begin{array}{ll}\{\top\}, & b_j=0\\ \rm{Free}_{\bb{V}^\ast}(k_j), & b_j=1 \end{array}\right. .$$
\end{theorem}

\subsection{Product algebras and product triples}\label{product_algebras}

Product algebras are the algebraic counterpart of product fuzzy logic (\cite{CT}). They are commutative integral bounded residuated lattices which are also prelinear, divisible and satisfy the extra equation:
\begin{equation}\label{eq:producto}
\neg x \vee ((x\to (x\ast y))\to y)=\top
\end{equation}
Since these algebras are pseudocomplemented (i.e., they satisfy the equation $x\wedge \neg x=\bot$) and prelinear, they form a proper subvariety $\mathbb{P}$ of  Stonean residuated lattices.

Following Example \ref{product_example} we already know that the restriction of the functor ${\mbf T}$ to $\mathbb{P}$ defines a categorical equivalence with the subcategory of triples $\sh{T}_{\mathbb{P}}.$

\medskip

In \cite{MontUgo} Montagna and Ugolini proved a categorical equivalence between $\mathbb{P}$ and a category of \textit{product triples} $(\mbf B, \mbf C, \vee_e)$ such that $\mbf B$ is a Boolean algebra, $\mbf C$ a cancellative hoop and $\vee_e:\B\times \mbf C\to \mbf C$ a function satisfying
\begin{itemize}
	\item[(V1)] For fixed $b\in B, c\in C$, $(b\vee_e\cdot):\mbf C\to\mbf C$ is a cancellative hoop morphism and $(\cdot\vee_e c):\mbf B\to \mbf C$ is a lattice morphism.
	\item[(V2)] For $b\in B, c\in C$, $0\vee_e c= c$ and $1\vee_e c=1$.
	\item[(V3)] For $b,b'\in B, c,c'\in C$, $$(b\vee_e c)\vee (b'\vee_e c')=(b\vee b')\vee_e(c\vee c')=b\vee_e(b'\vee_e(c\vee c')).$$
	\item[(V4)] For $b\in B, c,c'\in C$, $(b\vee_e c)c'=(\neg b\vee_e c')\wedge (b\vee_e (cc'))$
\end{itemize}

Given a product algebra $\mbf P$, its associated product triple is $$(\mbf B(P), \mbf C(P), \vee_e),$$ where $\mbf B(P)$ is the algebra of Boolean elements of $\mbf P$, $\mbf C(P)$ the greatest cancellative hoop of $\mbf P$ (that coincides with the algebra of dense elements $\mbf D(P)$) and $\vee_e$ is the supremum restricted to $B(P)\times C(P)$.

Given a product triple $(\mbf B, \mbf C, \vee_e)$, one can recover the product algebra defining appropriately the algebra operations in the set $(B\times C)/\sim$, where $\sim$ is the equivalence relation given by
\begin{align*}(b,c)\sim(b',c')\text{ if and only if } b=b' \text{ and } \neg b\vee_e c=\neg b\vee_e c'.\end{align*}

The algebra thus obtained is denoted $\mbf B(P) \otimes_{\vee_e}\mbf C(P)$ and is isomorphic to $\mbf P$.

\medskip

Then we have two similar categorical equivalences for product algebras, both given by triples whose first two coordinates coincide. We will now sketch how to go from one to the other.

Let $R=(\B, \D,\vee_e)$ be a product triple as defined by Montagna and Ugolini. Define $$\phi_R(b)=\{d\in \D: d=\neg b\vee_e d\}.$$  Then $(\B, \D,\phi_R)$ is an object of $\sh{T}_{\mathbb{P}}.$  If ${\mbf A}$ is a product algebra and $R=(\B, \D,\vee_e)$ is its corresponding product triple, then it is not hard to check that $\T({\mbf A})=(\B, \D,\phi_R)$.

On the other hand, given an element $Q=(\B, \D,\phi)$ in the category $\sh T$ such that $\D$ is cancellative, divisible and prelinear, we will proceed to define $\vee_Q:\B\times \D \to \D$.

Let ${\mbf A}\in \bb{RL}$. We say  that  $G \in \mathcal{F}_l(\mathbf{A})$ is a \textit{central filter} provided that there is $G^\prime \in \mathcal{F}_l(\mathbf{A})$ such that for every  filter $F \in \mathcal{F}_l(\mathbf{A})$ the following equalities hold:
\bg{equation}\label{eq:centralfilter} F = (F \cap G) \lor (F \cap G^\prime) = (F \lor G) \cap (F \lor G^\prime).\end{equation}
Central filters are the central elements of the lattice $\mathcal{F}_l(\mathbf{A})$, i.~e., the  neutral and complemented elements (see \cite[Theorem(4.13)]{MaeMae}). Hence we can easily adapt the results that Gr\"{a}tzer gives on lattice ideals in  \cite[Page 152]{Gra78} to prove the following lemma:
\bg{lemma}\label{gratzer} Let $\A \in \bb{RL}$ and let $G$ be a central filter of $\cc{F}_l(\A)$. For each $x \in A$ there is a unique $z \in A$ such that $[x) \cap G = [z)$.\e{lemma}

For the triple $Q$, since $\D$ is distributive $\phi(b)$ is clearly central in $\mathcal{F}_l(\mathbf{D})$ for each $b\in\mbf B$ (take $G^\prime=\phi(\neg b)$ and see Remark \ref{lattice_filter_dist}), so it follows from  Lemma~\ref{gratzer} that for each $b \in \mbf B$ we can define a function $$\fun{\rho_{b}}{\D}{\phi(b)}$$ by the prescription $$z = \rho_{b}(d) \,\,\, \mbox{ if and only if } \,\,\,[d) \cap \phi(b) = [z).$$ As a matter of fact, if $Q={\mbf T}(\mbf A)$ for some product algebra $\mbf A$ it is easy to see that $\rho_{b}(d) = \neg b \lor d = b\to d$, but note that $\rho_b$ can be expressed in terms of only the triple $(\B, \D,\phi)$ (cf.\cite{CheGra1},\cite[II.\S6 Theorem 5]{Gra78}).

Thus we can define $\vee_Q:\B\times \D \to \D$ by
$$b\vee_Q d=\rho_{\neg b}(d).$$ We leave as an exercise to check that $(\B, \D, \vee_Q)$ is the product triple associated to $\A.$

\medskip

The results of \cite{MontUgo} are generalized in \cite{AgFlUgo} for some subvarieties of MTL-algebras. We recall that MTL-algebras are the algebraic counterpart of the logic corresponding to monoidal t-norms (MTL, see \cite{EsGo}). The authors prove a categorical equivalence between strongly perfect MTL-algebras and a category of triples whose first component is a Boolean algebra, the second is a prelinear semihoop and the third is the extra operator $\vee$ as defined for product triples. Pseudocomplemented  strongly perfect MTL-algebras are Stone residuated lattices and prelinear Stone residuated lattices are strongly perfect MTL-algebras. So one can extend the previous comparison for the case of the intersection of the varieties of Stone residuated lattices and strongly perfect MTL-algebras.

%%%%%%%%%%%%%%%%%%%%%%%%%%%%%%%%%%%%%%%%%%%%%%%%%%%%%%%%%%%%%%%%%%%%%%%%%%%%%%%%
\section*{Acknowledgements}
%%%%%%%%%%%%%%%%%%%%%%%%%%%%%%%%%%%%%%%%%%%%%%%%%%%%%%%%%%%%%%%%%%%%%%%%%%%%%%%%

The authors would like to thank Sara Ugolini for the remark that helped us prove Theorem \ref{f_morphism} without using distributivity.

%%%%%%%%%%%%%%%%%%%%%%%%%%%%%%%%%%%%%%%%%%%%%%%%%%%%%%%%%%%%%%%%%%%%%%%%%%%%%%%%

Manuela Busaniche\\
Instituto de Matem\'{a}tica Aplicada del Litoral, UNL, CONICET, FIQ\\
Predio Dr. Alberto Cassano del CCT-CONICET-Santa Fe\\
Colectora de la Ruta Nacional no. 168\\
Santa Fe, Argentina\\
\texttt{mbusaniche@santafe-conicet.gov.ar}

\medskip

Roberto Cignoli\\
Universidad de Buenos Aires\\
Facultad de Ciencias Exactas y Naturales\\
Departamento de Matem\'{a}tica\\
Ciudad Universitaria\\
1428 Buenos Aires, Argentina.\\
\texttt{cignoli@dm.uba.ar}

\medskip

Miguel Andr\'es Marcos\\
Instituto de Matem\'{a}tica Aplicada del Litoral, UNL, CONICET, FIQ\\
Predio Dr. Alberto Cassano del CCT-CONICET-Santa Fe\\
Colectora de la Ruta Nacional no. 168\\
Santa Fe, Argentina\\
\texttt{mmarcos@santafe-conicet.gov.ar}

  \e{document}